\documentclass[12pt,a4]{amsart}

\usepackage{soul}
\usepackage[normalem]{ulem}

\usepackage[usenames]{color}
\usepackage{amsmath,amssymb,amsthm,amsfonts,enumerate,url,mathbbol,mathrsfs,tikz,graphicx}
\usepackage{hyperref}

\usetikzlibrary{automata,positioning}

\usepackage[utf8]{inputenc}

\usepackage{verbatim}
\usepackage{array}

\def\mG{\mathsf{G}}	

\def\mV{\mathsf{V}}

\def\mE{\mathsf{E}}

\def\mv{\mathsf{v}}
\def\mw{\mathsf{w}}

\def\me{\mathsf{e}}
\def\uu{\mathfrak{u}}
\def\vv{\mathfrak{v}}

\DeclareMathOperator{\Id}{Id}

\newtheorem{lemma}{Lemma}[section]

\newtheorem{theorem}[lemma]{Theorem}
\newtheorem{rem}[lemma]{Remark}
\newtheorem{proposition}[lemma]{Proposition}
\newtheorem{ex}[lemma]{Example}
\newtheorem{assums}[lemma]{Assumptions}

\theoremstyle{definition}

\theoremstyle{definition}
\numberwithin{lemma}{section}

\newcommand{\bound}{\mathcal{L}}
\newcommand{\A}{\mathcal{A}}
\newcommand{\h}{\widetilde{H}}
\newcommand{\R}{\mathbb{R}}
\newcommand{\C}{\mathbb{C}}

\newcommand{\G}{\mathcal{G}}

\newcommand{\N}{\mathbb{N}}

\newcommand{\f}{\mathfrak{f}}
\newcommand{\g}{\mathfrak{g}}

\numberwithin{equation}{section}

\usepackage{mathtools}

\date{}
\title[higher-order operators on networks: hyperbolic and parabolic theory]{higher-order operators on networks:\\ hyperbolic and parabolic theory}
\author{Federica Gregorio}
\address{Federica Gregorio, Dipartimento di  Matematica, Università degli Studi di Salerno, Via Giovanni Paolo II, 132, 84084 Fisciano (Sa), Italy}
\email{fgregorio@unisa.it}

\author{Delio Mugnolo}
\address{Delio Mugnolo, Lehrgebiet Analysis, Fakult\"at Mathematik und Informatik, Fern\-Universit\"at in Hagen, D-58084 Hagen, Germany}
\email{delio.mugnolo@fernuni-hagen.de}

\thanks{
The first author is member of the Gruppo Nazionale per l'Analisi Matematica,
la Probabilit\`a e le loro Applicazioni (GNAMPA) of the Istituto Nazionale di Alta Matematica
(INdAM). The second author has been partially supported by the Deutsche Forschungsgemeinschaft (Grant 397230547). Both authors would like to acknowledge networking support by the COST Action CA18232.}


\begin{document}

\maketitle

\begin{center}
\textit{Dedicated to Professor Silvia Romanelli on the occasion of her $70th$ birthday}
\end{center}

\begin{abstract}
We study higher-order elliptic operators on one-dimensional ramified structures (networks). We  introduce a general variational framework for fourth-order operators that allows us to study features of both hyperbolic and parabolic equations driven by this class of operators. We observe that they extend to the higher-order case and discuss well-posedness and conservation of energy of beam equations, along with regularizing properties of polyharmonic heat kernels. A noteworthy finding is the discovery of a new class of well-posed evolution equations with Wentzell-type boundary conditions. 
\end{abstract}

\section{Introduction}

Double-beam systems are a classical subject of theoretical mechanics, see e.g.~\cite{CheShe95,VuOrdKar00}: they consist of two beams mediated by a viscoelastic material layer. On the mathematical level, this is modeled by strong couplings between the equations, usually complemented by identical, homogeneous boundary conditions -- say, clamped or hinged. In the last two decades, coupled systems consisting of \textit{networks} of (almost) one-dimensional  beams have aroused more and more interest: unlike in double-beam systems, all interactions take place in the ramification points.

Our aim in this note is twofold: we first discuss some properties of beam equations
\begin{equation*}\label{eq:beam}
\frac{\partial^2 u}{\partial t^2}=-\frac{\partial^4u}{\partial x^4}
\end{equation*}
 on networks of one-dimensional  elements, with a focus on the solution properties that depend on rather general transmission conditions in the nodes.  To this purpose, we propose a variational treatment of beam equations  developing a formalism that   happens then to be easily extendible to the study of parabolic equations driven by elliptic operators of arbitrary even order, again with general combinations of stationary and dynamic boundary conditions.

The analysis of evolution equations on networks has become a very popular topic since Lumer introduced in~\cite{Lum80} a theoretical framework to study heat equations on ramified structures; but in fact, time-dependent Schrödinger equations on networks  have been studied by quantum chemists since the 1940s and perhaps earlier, see the references in~\cite[\S~2.5]{Mug14}. 
Also, networks of thin linear beams have been studied often in the literature, with a special focus on controllability and stabilization issues, ever since the pioneering discussion in~\cite{LagLeuSch92,LagLeuSch93}. 
The simplest case of a network corresponds to a path graph -- a concatenation of linear elements. This case models a beam consisting of different segments with various elasticity properties:   different vertex conditions for the bi-Laplacian that appears in the beam equation on a single path graph  have  been derived from physical principles in~\cite{CheDelKra87}. 
It turns out that there is no unique natural choice for transmission conditions in a network's node: based on physical considerations, several conditions have been proposed in the literature, especially of stationary nature~\cite{DekNic98,DekNic99,DekNic00,BorLaz04,DagZua06,KiiKurUsm15}. In~\cite{GreMug20} we have applied the classical extension theory of symmetric positive semidefinite operators to discuss general transmission conditions for the bi-Laplacian; in particular, we have described an infinite class of transmission conditions leading to well-posedness of the beam equation on networks. 

Along with stationary conditions, conditions of dynamic type have been very popular in the context of networks of beams, as they naturally model  
 massive junctions: we refer to \cite{MerReg08,MerReg08b,MerReg09,TorAfsNaj13}. 
 In Section~\ref{beameq} we are going to continue and extend the analysis of fourth-order operators on finite networks initiated in~\cite{GreMug20}: we characterize dynamic conditions leading to self-adjoint realizations and in fact, parametrize an infinite class of transmission conditions under which the corresponding system of beams is well-posed and enjoys conservation of energy. Rather general dynamic boundary conditions leading to self-adjoint, dissipative realizations on a single interval have been studied in~\cite{FavGolGol07}: they are special cases of our parametrization, too, which elaborates on an idea of Arendt and his co-authors~\cite{AreMetPal03,AreEls12} for the discussion of second-order elliptic operators with dynamic boundary conditions through quadratic forms on product Hilbert spaces. An interesting feature of our theoretical framework is that, by appropriately choosing the product space, we can simultaneously treat evolution equations endowed by dynamic and/or stationary conditions. We mention that similar ideas and a comparable formalism have been successfully applied in~\cite{KraMugNic20} to study hyperbolic systems with dynamic boundary conditions.

In Section~\ref{sec:generalpoly} we also  extend most of the methods developed for the beam equations to the study of parabolic features of equations driven by $j$th powers of the Laplacian, again on networks. Linear and semilinear elliptic equations associated with such operators have been discussed often in the literature since a classical article by Davies~\cite{Dav95}; we also refer  to \cite{BarGaz13} for a collection of the main features of such equations, to \cite{LANGER199973} for a deep study of contractivity properties of the generated semigroups,  to~\cite{GazGruSwe10} for a selection of related models in physics and mechanics, and to~\cite{FavGolGol08,FavGolGol10,DenKunPlo20} for a study of some realizations with dynamic boundary conditions on domains. After briefly showing well-posedness of general hyperbolic equations (second derivative in time, arbitrary integer powers of the Laplacians in space) with dynamic and/or stationary equations we turn to the properties of the analytic semigroup generated by the same differential operator's realizations on a finite network. In  particular, we show that such a semigroup is of trace class and (under mild assumptions) ultracontractive; we also show that in spite of failure of the maximum principle for any $j\ge 2$, depending on the transmission conditions such a semigroup may or may not be eventually sub-Markovian and eventually enjoy a strong Feller property.

Based on a classical idea that goes back to~\cite{FavGolGol02}, it is well-known that in the case of second-order elliptic operators there is a direct connection between dynamic and Wentzell-type  boundary conditions: formally taking the boundary trace of the evolution equation and plugging it into the dynamic boundary condition, one thus obtains a boundary condition involving boundary terms of order as high as the operator itself. 
We are also going to demonstrate that the class of Wentzell-type boundary conditions hitherto considered in the literature is unnecessarily restrictive; possibly the most surprising finding of our paper is that the natural Wentzell-type boundary conditions for an evolution equation of order $2j$ in space are in fact of \textit{higher} order than the operator itself: namely, of order $3j-1$, see Proposition~\ref{prop:surpr}.(iv).

\section{The beam equation on networks}\label{beameq}

We consider a finite connected graph $\mG=(\mV,\mE)$, with $V:=|\mV|$ vertices and $E:=|\mE|$ edges; loops and parallel edges are allowed. 
We also denote by $\mE_{\mv}$ the set of all edges incident in $\mv$. We fix an arbitrary orientation of $\mG$, so that each edge $\me\equiv(\mv,\mw)$ can be identified with an interval $[0,\ell_\me]$ and its endpoints $\mv,\mw$ with $0$ and $\ell_\me$, respectively. In such a way one naturally turns the $\mG$ into a metric measure space $\mathcal G$: a \textit{network} (or \textit{metric graph}) whose underlying discrete graph is precisely $\mG$. We refer to~\cite[Chapt.~3]{Mug14} for details.

Functions on $\mathcal G$ are vectors $(u_\me)_{\me\in\mE}$, where each $u_\me$ is defined on the edge $\me\simeq(0,\ell_\me)$. We introduce the Hilbert space of measurable, square integrable functions on $\mathcal G$
\[
\begin{split}
L^2(\mathcal G)&:=\bigoplus_{\me\in \mE} L^2(0,\ell_\me)\\
\end{split}
\]
endowed with the natural inner product 
\[
(u,v)_{L^2(\mathcal G)}:=\int_\G u(x)\overline{v(x)}\,dx= \sum\limits_{\me\in\mE}\int_0^{\ell_\me} u_\me(x)\overline{v_\me(x)}\,dx.
\]
Boundary values of elements of $L^2(\mathcal G)$ are not defined, and in this sense functions that are merely in $L^2(\mathcal G)$ cannot mirror the topology of the network $\mathcal G$: in order to describe transition conditions in the vertices we  to introduce the Sobolev spaces
\[\h^k(\mathcal G):=\bigoplus_{\me\in\mE} H^k(0,\ell_\me),\qquad k\in \mathbb N:\]
they consist of $L^2(\mathcal G)$-functions  whose $k$-th weak derivatives are elements of $L^2(\mathcal G)$, too.

Consider the operator $A$ defined  edgewise as the fourth derivative
\[
A:(u_\me)_{\me\in\mE}\mapsto (u''''_\me)_{\me\in\mE}
\]
(here and in the following, $'=\frac{d }{d x}$)  with domain 
\[
\bigoplus_{\me\in\mE}C_c^\infty(0,\ell_\me):
\]
it is symmetric and strictly positive, hence its self-adjoint extensions can be described by means of the extension theory due to Friedrichs and Krein. An important role is played by the closable quadratic form associated with $A$, which is given by
\begin{equation*}\label{formL2}
a(u,v)=\sum\limits_{\me\in\mE}\int_0^{\ell_\me}u_\me''(x)\overline{v_\me''(x)}\,dx,\qquad u,v\in \bigoplus_{\me\in\mE}C_c^\infty(0,\ell_\me).
\end{equation*}
However, a sesquilinear form can -- and typically will -- have	 different associated operators whenever it is studied on different Hilbert spaces. In~\cite[\S~3]{GreMug20} we have characterized the self-adjoint extensions of $A$ on $L^2(\mathcal G)$ and discussed further realizations that generate cosine operator functions and operator semigroups, again on $L^2(\mathcal G)$. In this paper we are going to discuss the more general case of extensions on Hilbert spaces of the form
\[
L^2(\mathcal G)\oplus Y_d
\]
where $Y_d$ is any subspace of the ``boundary space'' $\C^{4 E}$. Therefore, let us consider the space $L^2(\G)\oplus Y_d$ whose elements are of the form $\uu=\begin{pmatrix}u\\\theta\end{pmatrix}$. This is  a Hilbert space with respect to the canonical inner product
\[
\left(\mathfrak u,\mathfrak v\right)=\left(\begin{pmatrix}u\\ \theta\end{pmatrix},\begin{pmatrix}v\\ \phi\end{pmatrix}\right):=(u,v)_{L^2(\mathcal G)}+\left( \theta,\phi\right)_{Y_d},
\]
where $(\cdot,\cdot)_{Y_d}$ is the restriction to $Y_d$ of the canonical inner product of $\C^{4E}$.

In the following, we denote by $P_{Z}$ and $P^\perp_Z$ the orthogonal projector of $\C^{4E}$ onto subspaces $Z$ and $Z^\perp$, respectively;  we also introduce the notation
\[
\Gamma_\circ u:
=\begin{pmatrix}
\left(u_\me(0)\right)_{\me\in\mE}\\\left(u_\me(\ell_\me)\right)_{\me\in\mE}\\-\left(u'_\me(0)\right)_{\me\in\mE}\\\left(u'_\me(\ell_\me)\right)_{\me\in\mE}
\end{pmatrix}\quad\hbox{and}\quad
\Gamma^\circ u:=
\begin{pmatrix}
-\left(u'''_\me(0)\right)_{\me\in\mE}\\\left(u'''_\me(\ell_\me)\right)_{\me\in\mE}\\-\left(u''_\me(0)\right)_{\me\in\mE}\\-\left(u''_\me(\ell_\me)\right)_{\me\in\mE}
\end{pmatrix}.
\]

Consider the quadratic form 
\begin{equation}\label{formcont}
\begin{split}
a(\uu,\vv)&=\sum_{\me\in\mE}\int_0^{\ell_\me}u_\me''(x)\overline{v_\me''(x)}\,dx,\\
D(a)&=\left\{ \uu=\begin{pmatrix}u\\\theta\end{pmatrix}\in \h^2(\G)\oplus Y_d: \Gamma_\circ u=\theta\right\}.
\end{split}
\end{equation}
It is closable and its closure is associated with a self-adjoint operator: it is easily seen
  that this is the operator matrix
\begin{equation}\label{eq:aopmat}
\A=\begin{pmatrix}
\frac{d^4}{dx^4}&0\\-\Gamma^\circ&0
\end{pmatrix}
\end{equation}
with domain 
\[D(\A)=\left\{\begin{pmatrix}u\\\theta\end{pmatrix}\in \h^4(\G)\oplus Y_d:\Gamma_\circ u=\theta\right\}.\]
 Consider now $\A_{\max}$ and $\A_0$, the maximal and the minimal realizations of the  operator $\A$, respectively, endowed with the domains
\begin{align*}
D(\A_{\max})&=\h^4(\G)\oplus Y_d;\\
D(\A_0)&=\left\{\begin{pmatrix}
u\\ \theta
\end{pmatrix}\in \h^4(\G)\oplus Y_d: P_{Y_d}\Gamma_\circ u=\theta,\  P^\perp_{Y_d}\Gamma_\circ u=0,\hbox{ and } \Gamma^\circ u=0 \right\}.
\end{align*}
The main result in this section is a characterization of all further self-adjoint extensions of $\mathcal A_0$ on 
\[
L^2(\mathcal G)\oplus Y_d
\]
in the spirit of~\cite[Thm.~1.4.4]{BerKuc13}, see also~\cite[Thm.~3.1]{GreMug20} for the extension theory of bi-Laplacians on $L^2(\mathcal G)$.

Our starting point is the sesquilinear form \eqref{formcont}, which can be further generalized by adding a boundary term of the form $\left(R  \Gamma_\circ u,\Gamma_\circ v \right)$ for some $R\in{\mathcal L}(\C^{4E})$. Now, take any subspace $Y_s$ of $\C^{4E}$ that is orthogonal to $Y_d$
 and impose the boundary conditions
\begin{equation*}\label{eq:bc1}
\Gamma_\circ u,\Gamma_\circ v\in Y:=Y_d\oplus Y_s.
\end{equation*}
This motivates us to introduce the Hilbert space 
\[
\mathcal V:=\left\{\begin{pmatrix}u\\\theta\end{pmatrix}\in \widetilde{H}^2(\G)\oplus Y_d:\Gamma_\circ u\in Y\hbox{ and } P_{Y_d}\Gamma_\circ u=\theta\right\}.
\]

Hence, consider the sesquilinear form $\tilde{a}$ defined by
\begin{equation*}\label{form}
\begin{split}
\tilde{a}(\uu,\vv):&=\int_{\G} u'' \overline{v''}\, dx-\left(R   \Gamma_\circ u,\Gamma_\circ v \right)_{Y}\\
&=\int_{\G} u'' \overline{v''}\, dx-\left(P_{Y_s}R P_{Y_s}  \Gamma_\circ u,P_{Y_s}\Gamma_\circ v \right)_{Y_s}-\left(P_{Y_d}R P_{Y_d}  \Gamma_\circ u,P_{Y_d}\Gamma_\circ v \right)_{Y_d}
\end{split}\qquad \uu,\vv\in \mathcal V.
\end{equation*}
If additionally $u\in \h^4(\mathcal G)$, integrating by parts we find
\[
\int_\G u'' \overline{v''}\, dx=\int_\G  u'''' \overline{ v}\, dx -\left(\Gamma^\circ u,\Gamma_\circ v\right)_Y.
\]
We deduce for all $\uu,\vv\in  \mathcal V$ such that $u\in \h^4(\G)$
\begin{equation*}\label{eq:calc-fed-0}
\begin{split}
\tilde{a}(\uu,\vv)&=
\int_\G  u'''' \overline{ v}\, dx -\left(\Gamma^\circ u,\Gamma_\circ v\right)_Y -\left(R \Gamma_\circ u,\Gamma_\circ v \right)_Y
\end{split}
\end{equation*}
and hence
\begin{equation}\label{eq:calc-fed}
\tilde{a}(\uu,\vv)=
\int_\G  u'''' \overline{ v}\, dx -\left(P_{Y_d}(\Gamma^\circ u+RP_{Y_d}\Gamma_\circ u),P_{Y_d}\Gamma_\circ v\right)_{Y_d}-\left(P_{Y_s}(\Gamma^\circ u+RP_{Y_s}\Gamma_\circ u),P_{Y_s}\Gamma_\circ v\right)_{Y_s}.
\end{equation}
If we additionally impose $(\Gamma^\circ u+RP_{Y_s}\Gamma_\circ u)\perp Y_s$, i.e.,
\begin{equation}\label{eq:bc2}
P_{Y_s}\left(\Gamma^\circ u+RP_{Y_s}\Gamma_\circ u\right)=0,
\end{equation}
we can hence compactly write 
\begin{equation*}\label{eq:compact}
\tilde{a}(\uu,\vv)=\left(\begin{pmatrix}
\frac{d^4}{dx^4} & 0\\
-P_{Y_d} \Gamma^\circ  & -P_{Y_d}R
\end{pmatrix}
\begin{pmatrix}
u\\ P_{Y_d}\Gamma_\circ u
\end{pmatrix} , \begin{pmatrix}
v\\ P_{Y_d}\Gamma_\circ v
\end{pmatrix}\right)_{L^2(\mathcal G)\oplus Y_d},
\end{equation*}
 for all $\uu\in (\widetilde{H}^4(\G)\oplus Y_d)\cap \mathcal V$ satisfying \eqref{eq:bc2} and all $\vv \in \mathcal V$. Summing up, we describe the transmission conditions in the network's vertices by means of a subspace $Y$ of $\C^{4E}$: this consists of two orthogonal subspaces $Y_d$ and $Y_s$ that encode the dynamic and stationary part of the transmission conditions, respectively.

Observe that the quadratic form $\tilde{a}$ in \eqref{eq:calc-fed} is symmetric if and only if both $D:=P_{Y_d}RP_{Y_d}$ and $S:=P_{Y_s}RP_{Y_s}$ are self-adjoint. 

 Let us first focus on the case $D=0$.

\begin{theorem}\label{thm:extension}
 Let $Y_d$ be a subspace of $\C^{4E}$. Then,  for any extension $\mathcal A$ of $\mathcal A_0$ on $L^2(\mathcal G)\oplus Y_d$, the following are equivalent.
\begin{enumerate}[(i)]
\item $\mathcal A$ is self-adjoint.
\item The operator $\A$ has the form
\[
\begin{split}
\mathcal A&=\begin{pmatrix}\frac{d^4 }{dx^4} & 0 \\ -P_{Y_d}\Gamma^\circ  &0 \end{pmatrix},\\
D(\mathcal A)&=\bigg\{ \begin{pmatrix}u\\ {\bf \theta}\end{pmatrix}\in\h^4(\G)\oplus Y_d: \Gamma_\circ u\in Y,\ P_{Y_d}\Gamma_\circ u=\theta, \hbox{ and } P_{Y_s}(\Gamma^\circ u+SP_{Y_s}\Gamma_\circ u)=0\bigg\},
\end{split}
\]
 for some subspace $Y_s$ of $\C^{4E}$ orthogonal to $Y_d$ with $Y:=Y_d\oplus Y_s$ and some self-adjoint linear operator $S\in\mathcal{L}(Y_s)$. 

\end{enumerate}
\end{theorem}

In the case of $Y_d=\{0\}$, this has been proved in~\cite{GreMug20}. Theorem~\ref{thm:extension} sharpens the main result in~\cite{FavGolGol07} already in the case of an interval (i.e., a graph consisting of a unique edge).

\begin{proof}
(i)$\Rightarrow$(ii) 
Because $\mathcal A$ is an extension of the minimal realization $\mathcal A_0$, hence a restriction of the maximal realization $\mathcal A_{\max}$, $\mathcal A$ has the same form of the operator matrix in~\eqref{eq:aopmat}. 

Let $\uu\in D(\mathcal A)$ and $\vv=\begin{pmatrix}v\\\phi\end{pmatrix}\in D(\A_{\max})$:
self-adjointness of $\mathcal A$ amounts to the condition that
\begin{align*}
(\mathcal A\uu,\vv)_{L^2(\G)\oplus Y_d}&=\int_\G  u'''' \overline{ v}\, dx -(P_{Y_d}\Gamma^\circ u,\phi)_{Y_d}
\\
&=\int_\G u\overline{v''''}\,dx -(P_{Y_d}\Gamma^\circ u,\phi)_{Y_d}+(\Gamma^\circ u,\Gamma_\circ v)-(\Gamma_\circ u,\Gamma^\circ v)
\\
&=\int_\G u\overline{v''''}\,dx -(P_{Y_d}\Gamma^\circ u,\phi)_{Y_d}+\left(P_{Y_d}\Gamma^\circ u,P_{Y_d}\Gamma_\circ v\right)_{Y_d}\\&\qquad-\left(P_{Y_d}\Gamma_\circ u,P_{Y_d}\Gamma^\circ v\right)_{Y_d}+\left(P_{Y_s}\Gamma^\circ u,P_{Y_s}\Gamma_\circ v\right)_{Y_s}-\left(P_{Y_s}\Gamma_\circ u,P_{Y_s}\Gamma^\circ v\right)_{Y_s}\\&\qquad+(P_{Y^\perp}\Gamma^\circ u,P_{Y^\perp}\Gamma_\circ v)_{Y^\perp}-(P_{Y^\perp}\Gamma_\circ u,P_{Y^\perp}\Gamma^\circ v)_{Y^\perp}
\end{align*}
agrees with
\begin{align*}
(\uu,\A\vv)_{L^2(\G)\oplus Y_d}&=\int_\G  u \overline{ v''''}\, dx -(P_{Y_d}\Gamma_\circ u,P_{Y_d}\Gamma^\circ v)_{Y_d}.
\end{align*}
Therefore, the boundary terms should vanish. Considering that $\Gamma_\circ u\in Y$ one has
\[
\left\{
\begin{split}
&(P_{Y_d}\Gamma^\circ u,\phi)_{Y_d}-\left(P_{Y_d}\Gamma^\circ u,P_{Y_d}\Gamma_\circ v\right)_{Y_d}
=0,\\
&(P_{Y^\perp}\Gamma^\circ u,P_{Y^\perp}\Gamma_\circ v)_{Y^\perp}=0 
 ,\\
&\left(P_{Y_s}\Gamma^\circ u,P_{Y_s}\Gamma_\circ v\right)_{Y_s}-\left(P_{Y_s}\Gamma_\circ u,P_{Y_s}\Gamma^\circ v\right)_{Y_s}=0.
\end{split}
\right.
\]
The first equality $(P_{Y_d}\Gamma^\circ u,\phi)_{Y_d}=(P_{Y_d}\Gamma^\circ u,P_{Y_d}\Gamma_\circ v)_{Y_d}$ shows that 
 $\phi=P_{Y_d}\Gamma_\circ v$. The remaining conditions are
 \begin{equation}\label{self-adj}
\left\{
\begin{split}
&(P_{Y^\perp}\Gamma^\circ u,P_{Y^\perp}\Gamma_\circ v)_{Y^\perp}=0 
 ,\\
&\left(P_{Y_s}\Gamma^\circ u,P_{Y_s}\Gamma_\circ v\right)_{Y_s}-\left(P_{Y_s}\Gamma_\circ u,P_{Y_s}\Gamma^\circ v\right)_{Y_s}=0,
 \end{split}
 \right.
 \end{equation}
from which it follows that  $S$ is self-adjoint and $\vv\in D(\A)$. Indeed, from the first condition of \eqref{self-adj} one straightforwardly obtains  that $\Gamma_\circ v\in Y$. From the last one, using the fact that  $P_{Y_s}(\Gamma^\circ u+SP_{Y_s}\Gamma_\circ u)=0$ one obtains      $\left(SP_{Y_s}\Gamma_\circ u,P_{Y_s}\Gamma_\circ v\right)_{Y_s}-\left(P_{Y_s}\Gamma_\circ u,P_{Y_s}\Gamma^\circ v\right)_{Y_s}=0$ and hence $S$ needs to be self-adjoint and $P_{Y_s}(\Gamma^\circ v+SP_{Y_s}\Gamma_\circ v)=0$.

(ii)$\Rightarrow$(i) 
In order to prove self-adjointness of $\A$ we have to establish two facts: (a) if $\uu$ and $\vv$ belong  to $D(\A)$, then \eqref{self-adj} holds and (b) if $\uu\in D(\A)$ and \eqref{self-adj} holds, then $\vv\in D(\A)$.
 If $\uu,\vv$ belong to $D(\A)$, then the set of equalities \eqref{self-adj} holds and this take cares of (a). If instead $\uu\in D(\A)$ and $\eqref{self-adj}$ holds with $S$ self-adjoint one directly has, as shown before, that $\vv\in D(\A)$ that is (b). 
\end{proof}

This motivates us to impose the following throughout this  section.

\begin{assums}\label{ass:main1}
$Y$ is a subspace of $\C^{4E}$, $Y_d,Y_s$ are  orthogonal subspaces of $Y$ such that $Y=Y_d\oplus Y_s$, and $S$ is  a linear operator on $Y_s$. 
\end{assums}

Let us recall a celebrated result due to J.\ Kisy\'nski: given a closed, densely defined operator $A$ on a Banach space $X$, generation of a cosine operator function by $A$ is equivalent to the existence of a space $V$ such that $D(A)\hookrightarrow V\hookrightarrow X$ and that the part of the operator matrix
\[
\begin{pmatrix}
0 & I\\ A & 0
\end{pmatrix}
\]
in $V\oplus X$ generates a strongly continuous semigroup, see~\cite[Thm.~3.14.11]{AreBatHie10}. In this case, $V$ is unique and is often called \textit{Kisy\'nski space} in the literature.

\begin{lemma}\label{lem:forma}
Under the Assumptions~\ref{ass:main1}, the operator $-\mathcal A$  associated with the form
\[a(\mathfrak u,\mathfrak v):=\int_\G  u'' \overline{ v''}\, dx -\left(SP_{Y_s}   \Gamma_\circ u,P_{Y_s}\Gamma_\circ v \right)_{Y_s}
\]
with domain
\[
\mathcal V:=\left\{\begin{pmatrix}u\\\theta\end{pmatrix}\in \widetilde{H}^2(\G)\oplus Y_d:\Gamma_\circ u\in Y,\ P_{Y_d}\Gamma_\circ u=\theta\right\}
\]
 generates on $L^2(\mathcal G)\oplus Y_d$ a cosine operator function 
with Kisy\'nski space $\mathcal V$. 
\end{lemma}
\begin{proof}
The sesquilinear form $a$ is the same form introduced in \cite{GreMug20},  whereas $\mathcal V$ is isomorphic to 
\[
\h^2_Y(\mathcal G):=\left\{u\in \widetilde{H}^2(\G):\Gamma_\circ u\in Y\right\}.
\]
We have already checked in \cite[Thm.~4.3]{GreMug20} that $a$ is densely defined and continuous. Let
\[
j:\h^2_Y(\mathcal G)\ni u\mapsto 
\begin{pmatrix}
u\\ P_{Y_d} \Gamma_\circ u
\end{pmatrix}
\in L^2(\mathcal G)\oplus Y_d:
\]
it is clear that this map has dense range, hence $a$ is a $j$-elliptic form in the sense of~\cite[\S~2]{AreEls12}, and the associated operator in the sense of~\cite[Thm.~2.1]{AreEls12} agrees with the operator associated with $a$ with domain $\mathcal V$.
Because
\[
|\Im a(\mathfrak u,\mathfrak u)|\leq c\Vert S\Vert_{\bound(Y_s)}\Vert u\Vert_{\widetilde{H}^2(\mathcal G)}\Vert j(\mathfrak u)\Vert_{L^2(\mathcal G)\oplus Y_d},
\]
the assertions then follows by a direct application of~\cite[Prop.~2.4]{MugNit12}.
\end{proof}

We are finally in the position to prove the main result of this section; we re-introduce the boundary term $D=P_{Y_d}RP_{Y_d}$ which we have been discussing at the beginning of this section, along with further perturbing terms.

\begin{theorem}\label{thm:main}
Under the Assumptions~\ref{ass:main1} 
let, for all $\me\in\mE,$  $ p_\me\in L^\infty(0,\ell_\me)$ be real-valued  such that $p_\me(x) \ge P_\me$ for some $P_\me>0$ and a.e.\ $x\in (0,\ell_\me)$ and let $\Pi$ be a self-adjoint, positive definite operator on $Y_d$. Then  for all $D\in\mathcal L(Y_d)$
\begin{equation}\label{eq:atilde}
-\tilde{\mathcal A}=\begin{pmatrix}-p\frac{d^4}{dx^4} & 0\\ \Pi P_{Y_d}\Gamma^{\circ}  & D\end{pmatrix}
\end{equation}
with domain
\begin{align*}
D(\tilde{\mathcal A})=\bigg\{ \begin{pmatrix}u\\ {\bf \theta}\end{pmatrix}\in\h^4(\G)\oplus Y_d: \Gamma_\circ u\in Y,\ P_{Y_d}\Gamma_\circ u={\bf \theta}\hbox{ and } P_{Y_s}\left(\Gamma^\circ u+SP_{Y_s}\Gamma_\circ u\right)=0\bigg\}
\end{align*}
generates on $L^2(\mathcal G)\oplus Y_d$ a cosine operator function 
with Kisy\'nski space $\mathcal V$.
\end{theorem}

\begin{proof}
Under our assumptions we can endow $L^2(\mathcal G)\oplus Y_d$ with the inner product
\begin{equation}\label{eq:new-inner}
\left(\begin{pmatrix}u\\ \theta\end{pmatrix},\begin{pmatrix}v\\ \phi\end{pmatrix}\right):=\sum\limits_{\me\in\mE}\int_0^{\ell_\me} \frac{1}{p_\me(x)}u_\me(x)\overline{v_\me(x)}\,dx+\left( \Pi \theta,\phi\right)_{Y_d}.
\end{equation}
This is again a Hilbert space and in fact the new inner product is equivalent to the canonical one, hence both Hilbert spaces are isomorphic. 

Let us first consider the case $D=0$. A direct computation similar to that preceding Theorem~\ref{thm:extension} shows that $\tilde{\mathcal A}$ is the operator associated with $a$ on $L^2(\mathcal G)\oplus Y_d$ with respect to the above inner product and we can prove just like in Lemma~\ref{lem:forma} that $-\tilde{\mathcal A}$ is the generator of a cosine operator function on  $L^2(\mathcal G)\oplus Y_d$ with respect to the canonical inner product, since so it is with respect to the equivalent inner product in~\eqref{eq:new-inner}.

In order to complete the proof, it suffices to observe that
the sesquilinear form
\[
b(\uu,\vv):=(D\theta,\phi),\qquad \begin{pmatrix}
u\\ \theta
\end{pmatrix},\begin{pmatrix}
v\\ \phi
\end{pmatrix}\in  L^2(\mathcal G)\oplus Y_d,
\]
is bounded.
Thus, the operator associated with the sesquilinear form $a+b$ in $L^2(\mathcal G)\oplus Y_d$ with respect to the inner product in~\eqref{eq:new-inner} -- i.e.\ $\tilde{\mathcal A}$ in \eqref{eq:atilde} -- is again the generator of a cosine operator function in $L^2(\mathcal G)\oplus Y_d$ with respect to the canonical inner product.
\end{proof}

\begin{rem}\label{rem:ref}
(1) We stress that while the case of $D\ne 0$ could also be dealt with as a bounded perturbation of a well-behaved operator, the case of $S\ne 0$ cannot and requires the specific treatment in Lemma~\ref{lem:forma}. 

(2) One sees that $\tilde{\mathcal A}$ is self-adjoint  -- or equivalently the associated sesquilinear form $\tilde{a}$, i.e.,
 \begin{equation}\label{eq:littlea-tilde}
\tilde a(\uu,\vv)=\int_{\G} u'' \overline{v''}\, dx-\left(S  P_{Y_s}\Gamma_\circ u,P_{Y_s}\Gamma_\circ v \right)_{Y_s}-\left(DP_{Y_d}   \Gamma_\circ u,P_{Y_d}\Gamma_\circ v \right)_{Y_d},
\qquad \uu,\vv\in \mathcal V,
\end{equation} 
 is symmetric -- if and only if $S,D$ are self-adjoint operators. Furthermore, 
 $\tilde{\mathcal A}$ is self-adjoint \textit{and} positive semi-definite -- or equivalently the associated sesquilinear form $\tilde{a}$ is symmetric \textit{and} accretive -- if $S,D$ are self-adjoint \textit{and} negative semi-definite operators; this condition is however not necessary, even in the simple case of $Y_d=\{0\}$, a counterexample being the Krein--von Neumann extension of $\Delta^2_{|C^\infty_c(0,1)}$ discussed in~\cite[Exa.~4.6]{GreMug20}.
\end{rem}

As a consequence of  Theorem~\ref{thm:main}, for all $\f\in D(\tilde\A)$ and all $\g\in\mathcal V$ there exists a unique solution
\[
\uu(t):=C(t,-\tilde\A)\f+S(t,-\tilde\A)\g
\] 
of the second-order abstract Cauchy problem associated  to $\tilde\A$  
\begin{equation}\label{eq:acp2-dyn}
\begin{cases}
\frac{\partial^2 \uu}{\partial t^2}(t,x)=-\tilde\A \uu(t,x), &t\ge 0,\, x\in \G,\\
\uu(0,x)=\f(x), &x\in\G,\\
\frac{\partial \uu}{\partial t}(0,x)=\g(x), &x\in\G,
\end{cases}
\end{equation}
where $(C(t,-\tilde\A))_{t\in \mathbb R}$ is the cosine operator function  generated by $-\tilde\A$ and $(S(t,-\tilde\A))_{t\in \mathbb R}$ denotes the sine operator function  generated by $-\tilde\A$, which is defined by 
\[
S(t,-\tilde\A)\f:=\int_0^t C(s,-\tilde\A)\f \ ds,\qquad t\in\R,\ \f\in L^2(\mathcal G)\oplus Y_d.
\]

Moreover, our approach based on forms and cosine operator functions allows us to derive the well-posedness of the damped wave equation 
\begin{equation}\label{eq:acp2-dyn-damp}
\begin{cases}
\frac{\partial^2 \uu}{\partial t^2}(t)=-\tilde\A \big(\uu(t)+\kappa\frac{\partial\uu}{\partial t}(t)\big), &t\ge 0,\\
\uu(0)=\f,\\
\frac{\partial \uu}{\partial t}(0)=\g,
\end{cases}
\end{equation}
for all $\kappa\in \C$: by~\cite[Thm.~3.1]{Mug08}, \eqref{eq:acp2-dyn-damp} is governed by an analytic semigroup of angle $\frac{\pi}{2}$.

\begin{rem}
The generation result in Theorem~\ref{thm:main} lies within the scope of~\cite[Thm.~5.3]{Mug06}. In practice, however, the conditions proposed there are very difficult to check, since cosine function generation by $-{\mathcal A}$ is shown to be equivalent to cosine function generation by a rather nasty perturbation of $-\frac{d^4}{dx^4}$ with transmission conditions $\Gamma_\circ u\in Y_s$; while the latter bi-Laplacian realization is well-behaved by the theory developed in~\cite{GreMug20}, unbounded perturbation theory for cosine operator function is a notoriously tricky business.

 Theorem~\ref{thm:main} can also be compared with some results in~\cite{FavGolGol07}. For example, all cases where $Y_d\ne Y_1\oplus\{0_{\C^{2E}}\}$ are covered by Theorem~\ref{thm:main} but seemingly not by~\cite[Thm.~8]{FavGolGol07}.
 \end{rem}

Beam equations with dynamic boundary conditions have been often considered in the literature and can be studied with our formalism. In order to make this formalism more concrete we give some examples. 

\begin{ex}\label{esempio}
\begin{enumerate}
\item
Fix a vertex $\mv_1$ and consider the beam equation with  the following transmission conditions:
\begin{itemize}
\item[(i)] continuity for $u$ and $u''$ in every vertex $\mv\in\mV$; 
\item[(ii)]  Kirchhoff  condition for the normal derivative and third normal derivative in every vertex except for $\mv_1$;
\item[(iii)] dynamic condition in $\mv_1$: $\displaystyle
\frac{\partial^2 u}{\partial t^2}(\mv_1)=\sum_{\me\sim \mv_1} \frac{\partial u''_\me}{\partial \nu}(\mv_1);$

\item[(iv)] compatibility condition in $\mv_1$ :
$\displaystyle u''(\mv_1)=\sum_{\me\sim \mv_1} \frac{\partial u_\me}{\partial \nu}(\mv_1).$
\end{itemize}
 With the formalism of Theorem~\ref{thm:extension}, the above vertex conditions correspond to a second order abstract Cauchy problem~\eqref{eq:acp2-dyn}, where the operator $\tilde{\mathcal A}$ is determined by
 the choice of the subspaces
\begin{align*}
Y&=\bigoplus_{\mv\in\mV} \langle {\mathbf 1}_{\mE_\mv}\rangle\oplus \left(\C^{\deg\mv_1}\oplus \bigoplus_{\mv\ne\mv_1} \langle {\mathbf 1}_{\mE_\mv}\rangle^\perp\right),\\
 Y_d&=\left(\langle {\mathbf 1}_{\mE_{\mv_1}}\rangle\oplus\bigoplus_{\mv\ne\mv_1} \{0\}\right)\oplus \bigoplus_{\mv\in\mV} \{0\},\quad Y_s=\left(\{0_{\mv_1}\}\oplus \bigoplus_{\mv\ne\mv_1} \langle {\mathbf 1}_{\mE_\mv}\rangle\right)\oplus \left(\C^{\deg\mv_1}\oplus \bigoplus_{\mv\ne\mv_1} \langle {\mathbf 1}_{\mE_\mv}\rangle^\perp\right);
 \end{align*}
here ${\mathbf 1}_{\mE_\mv}$ denotes the characteristic function of the set $\mE_\mv$ of edges incident with any given vertex $\mv$.
\item C.\ Castro and E.\ Zuazua have discussed the boundary controllability of a non-degenerate (\cite{CasZua00}) and a degenerate (\cite{CasZua98}) equation modeling the vibrations of two flexible beams connected by a point mass: remarkably, the model in~\cite{CasZua98} involves \textit{two} dynamic conditions.
The transmission conditions they consider in the connecting vertex can be written in our formalism taking 
$D=S=0$ and $Y=Y_d= \langle {\mathbf 1}_{\mE_\mv}\rangle\oplus\langle {\mathbf 1}_{\mE_\mv}\rangle^\perp$ with $Y_s=\{0_{\C^{2}}\}\oplus\{0_{\C^{2}}\}$; and 
$D=S=0$ and $Y= \langle {\mathbf 1}_{\mE_\mv}\rangle\oplus\langle {\mathbf 1}_{\mE_\mv}\rangle^\perp$ with $Y_d=\langle {\mathbf 1}_{\mE_\mv}\rangle\oplus\{0_{\C^{2}}\}$ and $Y_s=\{0_{\C^{2}}\}\oplus\langle {\mathbf 1}_{\mE_\mv}\rangle^\perp$, respectively.
\end{enumerate}
\end{ex}

Let us now investigate on conservation of energy for the beam equation.
Let us  introduce the energy-type functionals in $L^2(\G)\oplus Y_d$
 \begin{equation*}\label{eq:energyd}
 K(t):=\frac12\left\| \frac{\partial \mathfrak u}{\partial t}(t)\right\|^2_{L^2(\mathcal G)\oplus Y_d},\qquad  P(t):=\frac{1}{2}\tilde{a}(\mathfrak{u}(t)),\qquad E(t):=K(t)+P(t)\ ,
 \end{equation*}
 for any solution $\mathfrak u=\begin{psmallmatrix}u\\ P_{Y_d}\Gamma_\circ u \end{psmallmatrix}$ of~\eqref{eq:acp2-dyn}, 
 where $\tilde{a}$ is the sesquilinear form in~\eqref{eq:littlea-tilde}.

\begin{lemma}\label{lem:energy}
Under the assumptions of Theorem~\ref{thm:main}, let $\tilde{\mathcal A}$ be self-adjoint and positive semi-definite. Then the total energy $E$ of~\eqref{eq:acp2-dyn} is conserved, i.e., it is a constant (over time) that only depends on the initial data $\f\in D(\tilde{\mathcal A}),\g\in\mathcal V$. 
\end{lemma}

 \begin{proof}
 Let us first observe that $$K(t)=\frac{1}{2}\int_\G \left(\frac{\partial u}{\partial t}(t,x)\right)^2\,dx+\frac12\left(\frac{d}{dt}\Gamma_\circ u,\frac{d}{dt}\Gamma_\circ u\right)_{Y_d}$$ and $$P(t)=\frac{1}{2}\int_\G \left(\frac{\partial^2 u}{\partial x^2}(t,x)\right)^2\,dx+\frac{1}{2}\frac{d}{d t}\Vert (-D)^{\frac{1}{2}}P_{Y_d}\Gamma_\circ u\Vert_{Y_d}^2+\frac{1}{2}\frac{d}{d t}\Vert (-S)^{\frac{1}{2}}P_{Y_s}\Gamma_\circ u\Vert_{Y_s}^2.$$
 
Differentiating the energy of a given solution $\mathfrak u$ with respect to $t$ and integrating by parts one obtains 
\begin{equation*}
\begin{split}
\frac{dE}{dt}(t)&=\int_\G \left(\frac{\partial u}{\partial t}\frac{\partial^2 u}{\partial t^2}+\frac{\partial^2 u}{\partial x^2}\frac{\partial^3 u}{\partial t\partial x^2}\right)\,dx+\left(\frac{d}{dt}\Gamma_\circ u,\frac{d^2}{dt^2}\Gamma_\circ u\right)_{Y_d}\\&\qquad+\frac{1}{2}\frac{d}{d t}\Vert (-D)^{\frac{1}{2}}P_{Y_d}\Gamma_\circ u\Vert_{Y_d}^2+\frac{1}{2}\frac{d}{d t}\Vert (-S)^{\frac{1}{2}}P_{Y_s}\Gamma_\circ u\Vert_{Y_s}^2\\
&=\int_\G \left(\frac{\partial u}{\partial t}\frac{\partial^2 u}{\partial t^2}+\frac{\partial u}{\partial t}\frac{\partial^4 u}{\partial x^4}\right)\,dx-\left(\frac{d}{dt} \Gamma_\circ u,\Gamma^\circ u\right)+\left(\frac{d}{dt}\Gamma_\circ u,\frac{d^2}{dt^2}\Gamma_\circ u\right)_{Y_d}\\&\qquad-\left(\frac{d}{dt}\Gamma_\circ u,DP_{Y_d}\Gamma_\circ u\right)_{Y_d}-\left(\frac{d}{dt}\Gamma_\circ u,SP_{Y_s}\Gamma_\circ u\right)_{Y_s}\\
&=\left(\frac{d}{dt}\Gamma_\circ u, -\Gamma^\circ u+\frac{d^2}{d t^2}\Gamma_\circ u-DP_{Y_d}\Gamma_\circ u\right)_{Y_d}-\left(\frac{d}{dt}\Gamma_\circ u, \Gamma^\circ u+SP_{Y_s}\Gamma_\circ u\right)_{Y_s}=0
\end{split}
\end{equation*}
since $P_{Y_d}\left(\frac{d^2}{d t^2}\Gamma_\circ u(t)-\Gamma^\circ u(t)-DP_{Y_d}\Gamma_\circ u(t)\right)=0$ for all $t\geq0$ and $P_{Y_s}(\Gamma^\circ u+SP_{Y_s}\Gamma_\circ u)=0$.
\end{proof}

This motivates us to introduce the notation 
\[
E(\f,\g)
\]
for the total energy of~\eqref{eq:acp2-dyn} with initial data $\f,\g$. J.A.\ Goldstein and coauthors have studied since~\cite{Gol69} whether wave-like equations enjoy \textit{equipartition of energy}, i.e., when
\begin{equation*}\label{equi}
\lim_{t\to \pm\infty} K(t)=\lim_{t\to \pm\infty} P(t)=\frac{1}{2}E(\f,\g)\quad \hbox{for all }\f\in D(\tilde{\mathcal A}),\ \g\in \mathcal V.
\end{equation*} 

\begin{proposition}\label{prop:equi}
Under the assumptions of Theorem~\ref{thm:main}, let $\tilde{\mathcal A}$ be self-adjoint and positive semi-definite. Then equipartition of energy fails for \eqref{eq:acp2-dyn}.
\end{proposition}
\begin{proof}
Let $\mathcal B$ be the square root of $\tilde{\mathcal A}$; it is well-known that its domain agrees with $\mathcal V$.
It is known that \eqref{eq:acp2-dyn}  enjoys equipartition of energy if and only if
\begin{equation*}
\lim_{s\to\pm \infty}(e^{is\mathcal B}\mathfrak\phi,\mathfrak\phi)_{L^2(\mathcal G)\oplus Y_d}=0\quad\hbox{for all }\phi\in L^2(\mathcal G)\oplus Y_d,
\end{equation*}
see~\cite[Thm.\ and the text around (14)]{Gol69}.
Now, because $\mathcal G$ is finite, $\mathcal V$ is compactly embedded in $L^2(\mathcal G)\oplus Y_d$ and hence $\mathcal{B}$ has compact resolvent: accordingly, there exists an orthonormal basis of $L^2(\mathcal G)\oplus Y_d$ of eigenvectors of $\mathcal B$. Let $\lambda$ be an eigenvalue of $-\mathcal B$ and $\mathfrak\phi$ be a corresponding normalized eigenvector: then 
\[
(e^{is\mathcal B}\mathfrak\phi,\mathfrak\phi)_{L^2(\mathcal G)\oplus Y_d}=
e^{is\lambda}{\not\rightarrow} 0,
\]
showing that equipartition of energy fails to hold.
\end{proof}

In the proof of Proposition~\ref{prop:equi}, the square root $\mathcal B$ of our operator matrix $\tilde{\mathcal A}$ has appeared. While the proof relies on general properties of square roots, it is sometimes possible to describe it more closely: this is interesting because it delivers a more explicit formula for the cosine and sine operator functions generated by $-\tilde{\mathcal A}$: indeed, because $\mathcal B$ is self-adjoint, $i\mathcal B$ generates a unitary group. It is known that $C(t,-\tilde{\mathcal A})=\cosh(t,i\mathcal{B})$, i.e.,
\[
C(t,-\tilde{\mathcal A})=\frac{e^{it\mathcal B}+e^{-it\mathcal B}}{2},\qquad t\in \R,
\]
cf.~\cite[Exa.~3.14.15]{AreBatHie10}.

 \begin{ex}\label{exa:notsquare}
 \begin{enumerate}
 \item 
In~\cite{GreMug20} we have reviewed stationary transmission conditions that appear in several models of beam networks in the literature -- especially in~\cite{CheDelKra87,DekNic99,DekNic00,BorLaz04,KiiKurUsm15}, showing that they fit in our scheme; additionally, we have   considered the bi-Laplacian with continuity conditions across the vertices and zero conditions on the first, second, and third derivatives at the endpoints of each edge, and determined its Friedrichs and Krein--von Neumann extensions. All  these realizations satisfy the assumptions of Lemma \ref{lem:energy} and Proposition~\ref{prop:equi}, leading to conservation of energy and failure of equipartition of energy.
On the other hand, only the transmission conditions in~\cite[Exa.~3.2]{GreMug20}, taken from~\cite{DekNic99}, lead to a realization of the forth derivative that is a square operator.

\item 
A Laplacian realization on a network with conditions of continuity on each vertex $\mv_1,\ldots,\mv_n$, complemented by dynamic conditions in $\mv_1$ and (stationary) Kirchhoff conditions in $\mv_2,\ldots,\mv_n$ has been studied by the second author and S.\ Romanelli: we refer to~\cite{MugRom07} for more details and an overview of earlier appearances of this model in mathematical and biological literature. It can be written as
\begin{equation*}\label{eq:Bop} \mathcal{B}=\begin{pmatrix} -\frac{d^2}{dx^2}&0\\P_{\tilde Y_d}\gamma^\circ&0 \end{pmatrix} 
\end{equation*} 
with domain 
\begin{equation*}\label{eq:Bdom} D(\mathcal{B})=\left\{\begin{pmatrix}u\\\xi\end{pmatrix}\in\h^2(\G)\oplus \tilde Y_d:\gamma_\circ u\in \tilde Y,\ P_{\tilde Y_d}\gamma_\circ u=\xi \ \textrm{and}\ P_{\tilde Y_s} \gamma^\circ u=0 \right\}, \end{equation*}
with
\[
\tilde{Y}=\bigoplus_{\mv\in\mV} \langle {\mathbf 1}_{\mE_\mv}\rangle,\quad 
\tilde{Y}_d=\langle {\mathbf 1}_{\mE_{\mv_1}}\rangle\oplus\bigoplus_{\mv\ne\mv_1} \{ 0_\mv\},\quad 
\tilde{Y}_s=\{0_{\mv_1}\}\oplus \bigoplus_{\mv\ne\mv_1} \langle {\mathbf 1}_{\mE_\mv}\rangle,
\]
 where  \[ \gamma_\circ:u\mapsto \begin{pmatrix} (u_\me(0))_{\me\in\mE}\\(u_\me(\ell_\me))_{\me\in\mE} \end{pmatrix},\quad \gamma^\circ:u\mapsto \begin{pmatrix} -(u'_\me(0))_{\me\in\mE}\\(u'_\me(\ell_\me))_{\me\in\mE} \end{pmatrix}.\]
Taking the square of this operator leads to a bi-Laplacian realization that fits the scheme of our Theorem~\ref{thm:extension}. Indeed, the square of this Laplacian is precisely the bi-Laplacian presented in Example~\ref{esempio}.(1).
\end{enumerate}
 \end{ex}

\section{Parabolic theory of polyharmonic operators with boundary conditions on networks}\label{sec:generalpoly}

 In this section we are going to extend the theory developed in the previous section to the study of $j$th powers of the Laplacian, for generic $j\ge 2$. 
  It turns out that the formalism introduced before allows us to discuss parabolic problems driven by general poly-harmonic operators under very general (stationary or dynamic) boundary conditions.

It is easy to prove by induction that for all $j\in \N$
\begin{equation}\label{eq:gamma-ops-poly}
\begin{split}
\int_0^\ell (-1)^ju^{(2j)}(x) \overline{v(x)}dx -\int_0^\ell u^{(j)}(x) \overline{v^{(j)}(x)}\,dx&=\sum_{k=0}^{j-1} \left[(-1)^{j+k} \partial_\nu^{2j-k-1}u\right] \cdot \overline{\partial_\nu^{k}v}\\
&=: \Gamma^\circ u \cdot \overline{\Gamma_\circ v}\\
\end{split}
\end{equation}
where the vectors $\Gamma^\circ u,\Gamma_\circ v\in \C^{2j}$ are defined using the notation
\[
\partial^h_\nu u:=\begin{cases}
\begin{pmatrix}
u^{(h)}(0)\\u^{(h)}(\ell)
\end{pmatrix}\quad \hbox{if $h$ is even},\\
\begin{pmatrix}
-u^{(h)}(0)\\u^{(h)}(\ell)
\end{pmatrix}\quad \hbox{if $h$ is odd}.\\
\end{cases}
\]

This clearly suggests to introduce the sesquilinear form
\[
a(u,v):=\int_0^\ell u^{(j)}(x) \overline{v^{(j)}(x)}\,dx-(R\Gamma_\circ u,\Gamma_\circ v)_Y
\]
for all $u,v$ in the form domain
\[
\mathcal V:=\left\{u\in H^j(0,\ell):\Gamma_\circ u\in Y \right\}
\]
for any given subspace $Y$ of $\C^{2j}$ and any linear operator $R$ on $Y$. This form is symmetric (and hence the corresponding operator is self-adjoint) if and only if $R$ is self-adjoint; indeed, the corresponding operator $A$ is the operator $(-1)^j \frac{d^{2j}}{dx^{2j}}$ with boundary conditions
\[
\Gamma_\circ u\in Y,\qquad \Gamma^\circ u+R\Gamma_\circ u\in Y^\perp.
\]
Following the same ideas in the proof of~\cite[Thm.~1.4.4]{BerKuc13} one can prove that each self-adjoint realization of $(-1)^j \frac{d^{2j}}{dx^{2j}}$ is of this type.

Upon replacing $L^2(0,\ell)$ by $\bigoplus_{\me\in\mE}L^2(0,\ell_\me)\oplus Y_d$,  scalar-valued functions by $\C^{E}$-valued functions, and the boundary space $\C^{2j}$ by $\C^{2jE}$, we can consider the sesquilinear form
\[
\tilde a(\uu,\vv):=\int_{\mathcal G} u^{(j)}(x) \overline{v^{(j)}(x)}\,dx-\left(SP_{Y_s}  \Gamma_\circ u,P_{Y_s}\Gamma_\circ v \right)_{Y_s}-\left(DP_{Y_d}   \Gamma_\circ u,P_{Y_d}\Gamma_\circ v \right)_{Y_d},
\]
with domain
\begin{equation}\label{eq:kisypoly}
\mathcal V=\bigg\{ \begin{pmatrix}u\\ {\bf \theta}\end{pmatrix}\in\h^{j}(\G)\oplus Y_d: \Gamma_\circ u\in Y,\ P_{Y_d}\Gamma_\circ u=\theta\bigg\}
\end{equation}
and then extend the above considerations to the case of elliptic operators of order $2j$ on networks; the essential ideas coincide with those presented in the previous section and we omit the details.

\begin{theorem}\label{thm:extension-new-poly}
Let $j\in \N$ and $Y_d$ be a subspace of $\C^{2jE}$ and consider the operator
\[
\begin{split}
\mathcal A_0&=(-1)^j\begin{pmatrix} \frac{d^{2j}}{dx^{2j}} & 0\\ - P_{Y_d}\Gamma^{\circ}  & 0\end{pmatrix}\\
D(\A_0)&=\left\{\begin{pmatrix}
u\\ \theta
\end{pmatrix}\in \h^{2j}(\G)\oplus Y_d: P_{Y_d}\Gamma_\circ u=\theta,\  P^\perp_{Y_d}\Gamma_\circ u=0,\hbox{ and } \Gamma^\circ u=0 \right\},
\end{split}
\]
where the operators $\Gamma^\circ,\Gamma_\circ$ are defined in \eqref{eq:gamma-ops-poly}.
Then for any extension $\mathcal A$ of $\mathcal A_0$ on $L^2(\mathcal G)\oplus Y_d$, the following are equivalent.
\begin{enumerate}[(i)]
\item $\mathcal A$ is self-adjoint.
\item The operator $\A$ takes the form
\begin{align}
\mathcal A&=(-1)^j\begin{pmatrix} \frac{d^{2j}}{dx^{2j}} & 0\\ - P_{Y_d}\Gamma^{\circ}  & 0\end{pmatrix},\nonumber\\
D(\mathcal A)&=\bigg\{ \begin{pmatrix}u\\ {\bf \theta}\end{pmatrix}\in\h^{2j}(\G)\oplus Y_d: \Gamma_\circ u\in Y,\ P_{Y_d}\Gamma_\circ u=\theta, \hbox{ and } P_{Y_s}(\Gamma^\circ u+SP_{Y_s}\Gamma_\circ u)=0\bigg\},\label{eq:DApoly}
\end{align}
for some subspace $Y_s$ of $\C^{2jE}$ orthogonal to $Y_d$ and some self-adjoint linear operator $S$ on $Y_s$; here $Y:=Y_d\oplus Y_s$.
\end{enumerate}
\end{theorem}

In the non-dynamic case of $Y_d=\{0\}$, $\mathcal A_0$ satisfies zero boundary conditions on all derivatives up to order $2j-1$: hence $\mathcal A_0$ is a symmetric, positive definite operator and we recover the classical characterization of self-adjoint extensions of one-dimensional polyharmonic operators.

Motivated by the above result we therefore impose the following  in the remainder of this section.

\begin{assums}\label{ass:main2}
$j\in \mathbb N$, $Y$ is a subspace of $\C^{2jE}$, $Y_d,Y_s$ are orthogonal subspaces of $Y$ such that $Y=Y_d\oplus Y_s$, $S$ is  a linear operator on $Y_s$, and $D$ is  a linear operator on $Y_d$. 
\end{assums}

 We can thus state the following, without proof.

\begin{theorem}\label{thm:main-general}
 Under the Assumptions~\ref{ass:main2}, let $\Pi$ be a self-adjoint, positive definite operator on $Y_d$. Also,  let for all $\me\in\mE$  $ p_\me\in L^\infty(0,\ell_\me)$ be real-valued and such that $p_\me(x) \ge P_\me$ for some $P_\me>0$ and a.e.\ $x\in (0,\ell_\me)$. Then  for all $D\in\mathcal L(Y_d)$
\[-\tilde{\mathcal A}=(-1)^{j+1}\begin{pmatrix} p\frac{d^{2j}}{dx^{2j}} & 0\\ -\Pi P_{Y_d}\Gamma^{\circ}  & -D\end{pmatrix}\]
with domain $D(\mathcal A)$ as in~\eqref{eq:DApoly}
generates on $L^2(\mathcal G)\oplus Y_d$ a cosine operator function with Kisyński space $\mathcal V$ in~\eqref{eq:kisypoly}.
\end{theorem}

\begin{rem}
Theorem~\ref{thm:main-general} extends the generation results from~\cite{Mug10}, where only the case of $j=1$ and $Y_s=\{0\}$ was considered; in turn, the latter generalized the main assertions from~\cite{MugRom07}, where $Y=Y_d$ was taken to be the subspace of $\C^{2E}$ consisting of those vectors that are vertex-wise constant, i.e., $\bigoplus_{\mv\in\mV}\langle \mathbf 1_\mv\rangle$.
\end{rem}

By~\cite[Thm.~3.1.4.17]{AreBatHie10}, each generator of a cosine operator function also generates on the same Banach space an analytic semigroup of angle $\frac{\pi}{2}$.  In the next proposition we study some properties for this semigroup. For the sake of simplicity  we focus on the simple case of $p\equiv 1$ and $\Pi=\Id$, but see Remark~\ref{rem:general} below.

\begin{proposition} \label{prop:surpr}
Under the Assumptions~\ref{ass:main2} the following properties hold for the semigroup generated by
\[
-\tilde\A=(-1)^{j+1}\begin{pmatrix}
\frac{d^{2j}}{dx^{2j}}&0\\-P_{Y_d}\Gamma^\circ&-D\end{pmatrix}.
\]
\begin{enumerate}[(i)]
\item $e^{-t\tilde\A}$ is of trace class for all $t>0$.

\item If $D,S$ are dissipative, then there exist $C,\omega>0$ such that
\begin{equation*}
\Vert e^{-t\tilde{\A}}\Vert_{2\to\infty}\leq 
C(t^{-\frac{1}{4j}}
+1)
\qquad \hbox{for all }t>0,
\end{equation*}
 hence in particular $e^{-t\tilde\A}$ is ultracontractive.
\item $e^{-t\tilde\A}$ has for all $t>0$ an integral kernel of class $L^\infty$.

\item  The solution $\uu:=e^{-t\tilde\A}{\mathfrak f}$ of
\begin{equation*}\label{eq:acp2-dyn-poly}
\left\{
\begin{split}
\frac{\partial \uu}{\partial t}(t)&=-\tilde\A \uu(t), \quad t\ge 0,\\
\uu(0)&=\f
\end{split}
\right.
\end{equation*}
satisfies the boundary condition
\[
P_{Y_d}\Gamma_\circ \frac{\partial^{2j}u}{\partial x^{2j}}(t)+P_{Y_d}\Gamma^\circ u(t)+DP_{Y_d}\Gamma_\circ u(t)=0,\qquad t> 0.
\]

\end{enumerate}
\end{proposition}

The terms $\Gamma^\circ u$ and $\Gamma_\circ u$ involve differential terms of order up to $2j-1$ and $j-1$, respectively. The property in (iv) is therefore surprising:  it states that the natural order of the Wentzell-type boundary conditions for an operator of order $2j$ is not necessarily $2j$, as usually considered in the literature, but rather up to $3j-1$.

\begin{proof}
(i) Observe that the image of the form domain $\mathcal V$ under $j$ is compactly embedded in $L^2(\mathcal G)\oplus Y_d$, hence the operator $-\tilde\A$ has compact resolvent. Indeed, more is true: by~\cite{Gra68} the embedding of $j(\mathcal V)$ in $L^2(\mathcal G)\oplus Y_d$ is of Schatten class, hence the analytic semigroup generated by $-\tilde\A$ consists for all $t>0$ of trace class  operators~\cite[Rem.~3.4]{MugNit12}.

(ii) 
Let us then consider $L^\infty(\mathcal G)\oplus Y_d$ with the norm
\[\Vert\mathfrak u\Vert_\infty=\left\Vert \begin{pmatrix}
u\\\theta\end{pmatrix}\right\Vert_\infty:=\max\{\Vert u\Vert_{L^\infty(\G)},\Vert\theta\Vert_{Y_d}\}.
\]
Let ${\mathfrak u}:=\begin{psmallmatrix}u\\\theta\end{psmallmatrix}\in \mathcal V$.
The Gagliardo--Nirenberg inequality, cf.~\cite{Gag59}, yields
\[\Vert u\Vert_{L^\infty(\G)}\leq c_1\Vert u^{(j)}\Vert_{L^2(\G)}^\frac{1}{2j}\Vert u\Vert_{L^2(\G)}^\frac{2j-1}{2j}+c_2\Vert u\Vert_{L^2(\G)}.\]
On the other hand, because $\|u\|_{L^\infty(\G)}\le c\|u\|_{\widetilde{H}^j(\G)}$ we find that 
\begin{equation*}\label{bound}	\Vert\theta\Vert_{Y_d}=|\Gamma_\circ u|\leq \widetilde{c}_1\Vert u\Vert_{L^2(\G)}+\widetilde{c}_2\Vert u^{(j)}\Vert_{L^2(\G)}.\end{equation*}
Therefore, since $D$ and $S$ are dissipative
\begin{align*}\Vert\mathfrak{u}\Vert_\infty&\leq k_1\Vert u^{(j)}\Vert_{L^2(\G)}^\frac{1}{2j}\Vert u\Vert_{L^2(\G)}^\frac{2j-1}{2j}+k_2\Vert u\Vert_{L^2(\G)}\\&= k_1 \left(a(\mathfrak u)+(DP_{Y_d}\Gamma_\circ u,\Gamma_\circ u)_{Y_d}+(SP_{Y_s}\Gamma_\circ u,\Gamma_\circ u)_{Y_s}\right)^\frac{1}{4j}\Vert u\Vert_{L^2(\G)}^\frac{2j-1}{2j}+k_2\Vert u\Vert_{L^2(\G)}
\\&\leq k_1 \left(a(\mathfrak u)\right)^\frac{1}{4j}\Vert u\Vert_{L^2(\G)}^\frac{2j-1}{2j}+k_2\Vert u\Vert_{L^2(\G)}.
\end{align*}
Observe that $S,D$ dissipative also implies that the semigroup is contractive.  Also, recall that if an operator $A$ generates an analytic semigroup, then there exists a positive constant such that $||Ae^{-tA}||\leq \frac{c}{t}$ for all $t>0$. We follow an argument similar to that in \cite[Proposition 5.1]{GreMug20}: letting ${\mathfrak u}:=e^{-t\tilde\A}{\mathfrak f}$ and using analyticity and contractivity  on $L^2(\G)$ one obtains
\begin{equation}\label{eq:contra}
\begin{split}
\Vert e^{-t\tilde\A}\mathfrak f\Vert_{\infty}&\leq k_1\Vert\tilde\A e^{-t\tilde\A}\f\Vert_{L^2(\G)\oplus Y_d}^\frac{1}{4j}\Vert e^{-t\tilde\A}\f\Vert_{L^2(\G)\oplus Y_d}^\frac{1}{4j}\Vert e^{-t\tilde\A}\f\Vert_{L^2(\G)\oplus Y_d}^\frac{2j-1}{2j}+k_2\Vert e^{-t\tilde\A}\f\Vert_{L^2(\G)\oplus Y_d}\\
&\leq C(t^{-\frac{1}{4j}}+1)\Vert \f\Vert_{L^2(\G)\oplus Y_d}\\
&\leq \tilde{C}t^{-\frac{1}{4j}}e^{\omega t}\Vert \mathfrak f\Vert_{L^2(\G)\oplus Y_d}.
\end{split}
\end{equation}
This concludes the proof.

(iii)  The adjoint $\tilde\A^\ast$ of $\tilde\A$ satisfies the assumptions of this theorem, too, hence by (ii)
\[
\Vert e^{-t\tilde\A^\ast}\mathfrak f\Vert_{\infty}\leq Ct^{-\frac{1}{4j}}e^{\omega t}\Vert \mathfrak f\Vert_{L^2(\G)\oplus Y_d}
\]
and by duality
\begin{equation}\label{eq:contra-2}
\Vert e^{-t\tilde\A}\mathfrak f\Vert_{2}\leq C t^{-\frac{1}{4j}}e^{\omega t}\Vert \mathfrak f\Vert_{L^1(\G)\oplus Y_d}.
\end{equation}
By the semigroup law, combining~\eqref{eq:contra} and~\eqref{eq:contra-2} yields 
\begin{equation*}
\Vert e^{-t\tilde\A}\Vert_{1\to\infty}\leq C^2t^{-\frac{1}{2j}}e^{2\omega t}\qquad \hbox{for all }t>0,
\end{equation*}
and in particular $e^{-t\tilde\A}$ maps for all $t>0$ $L^1$ to $L^\infty$: the existence of an $L^\infty$-kernel then follows from the Kantorovich--Vulikh Theorem.

(iv) 
Because of the smoothing effect of the analytic semigroup generated by $-\tilde\A$, the solution $u(t,\cdot)$ is for all $t>0$ infinitely often differentiable (with respect to space), hence we can take the boundary values $\Gamma_\circ \frac{d^{2j}}{dx^{2j}}u$ of $\frac{d^{2j}}{dx^{2j}}u$. Because the time derivative and $\Gamma_\circ $ commute, plugging the parabolic equation satisfied in the interior of the edges into the dynamic boundary conditions we deduce that
\[(-1)^j \Gamma_\circ u^{2j}(t)=P_{Y_d}\Gamma^\circ u(t)+ DP_{Y_d}\Gamma_\circ u(t)\qquad t>0.\]
This concludes the proof.
\end{proof}

Observe that from the computations in \eqref{eq:contra} one also finds
\begin{equation*}\label{eq:ultra-non}
\|e^{-t\tilde\A}\f\|_{L^\infty(\G)\oplus Y_d}\leq c(t^{-\frac{1}{4j}}+1)^2\|\f\|_{L^1(\G)\oplus Y_d}\approx c\|\f\|_{L^1(\G)\oplus Y_d}\ \textrm{as}\ t\to\infty.
\end{equation*}

\begin{rem}\label{rem:general}
Parabolic equations driven by Laplacians on networks with dynamic vertex conditions have been studied in~\cite{MugRom07}. It has been observed in~\cite[Rem~3.6]{MugRom07} that modifying the coefficients of the normal derivative (the lower-left entry of the relevant operator matrix in that context) amounts to a relatively compact perturbation of an analytic semigroup generator: by a perturbation theorem due to Desch and Schappacher \cite[Thm.~3.7.25]{AreBatHie10}, the new operator generates an analytic semigroup, too. (Similar assertions were proved in~\cite{BanBelRei06,VazVit08}.)  The same idea carries over to our setting and yields that the operator matrix
\[-\tilde{\mathcal A}=(-1)^{j+1}\begin{pmatrix} p\frac{d^{2j}}{dx^{2j}} & 0\\ M  & T\end{pmatrix}\]
with domain
\begin{align*}D(\tilde{\mathcal A})=\bigg\{ \begin{pmatrix}u\\ {\bf \theta}\end{pmatrix}\in\h^{2j}(\G)\oplus Y_d: {\bf \theta}=P_{Y_d}\Gamma_\circ u\hbox{ and } P_{Y_s}\left(\Gamma^\circ u+SP_{Y_s}\Gamma_\circ u\right)=0\bigg\}
\end{align*}
for any bounded linear operators $M$ from $\h^{2j}(\G)$ to $Y_d$ and $T$ on $Y_d$, generates an analytic semigroup on $L^2(\mathcal G)\oplus Y_d$.
The proof of~\cite[Thm.~9]{BanBelRei06} can be modified to show that $-\tilde{\mathcal A}$ has a number of negative eigenvalues at least as large as the number of negative eigenvalues of $\Pi$, provided $M$ factorizes as $M=\Pi P_{Y_d}\Gamma^{\circ}$.
\end{rem}

 In~\cite{GreMug20} we have discussed the bi-Laplacian on $\G$ through extension theory of Hilbert spaces. We could pursue similar results here, but we avoid the details. Suffice it to say that if we impose continuity vertex conditions on the pre-minimal operator, i.e., we consider the operator matrix $\mathcal A$ in~\eqref{eq:aopmat} restricted to 
\[
\left\{\begin{pmatrix}u\\ \theta \end{pmatrix}\in C (\G)\oplus Y_d: u'\in \bigoplus_{\me\in\mE} C^\infty_c(0,\ell_\me)\hbox{ and } P_{Y_d}\Gamma_\circ u= \theta \right\},
\]
then its Friedrichs extension $\mathcal A_F$ is the realization of $\A$ whose domain contains functions $u$ that enjoy the boundary conditions $\partial_\nu^{h}u=0$ for all $1\le h\le j-1$, along with continuity of $u$ on the metric space $\G$ and a Kirchhoff-type condition on $\partial_\nu^{2j-1}u$ at each vertex. In particular, the null space of $\mathcal A_F$ is 1-dimensional: it is given by the space of all constants on $\G$. Also observe that $e^{-t\mathcal A_F}$ maps, for each $t>0$, $L^2(\G)\oplus Y_d$ into 
\[
D(\mathcal A_F)\hookrightarrow \left\{\begin{pmatrix}u\\ \theta \end{pmatrix}\in C (\G)\oplus Y_d:  u\in \bigoplus_{\me\in\mE} C^{2j-1}([0,\ell_\me])\hbox{ and }  P_{Y_d}\Gamma_\circ u=\theta \right\}.
\]
Combining these two facts with~\cite[Cor.~7.4 and Prop.~7.5]{GreMug20} we can immediately deduce remarkable properties of the semigroup generated by $-\mathcal A_F$, which we state without proof. It is interesting to compare them with the properties of the  $-\mathcal A_N$, defined as the realization of $-\mathcal A$ whose domain contains functions $u$ such that $\partial_\nu^{h}u$ is continuous of $\G$ for all $0\le h\le j-1$, while a Kirchhoff-type condition is satisfied by $\partial_\nu^{h}u$ for all $j\le h\le 2j-1$.
 \begin{proposition}
 Let $j\ge 2$.  Under the Assumptions~\ref{ass:main2}, for all subspaces $Y_d$ of $Y$ the semigroup on $L^2(\G)\oplus Y_d$ generated by $-\mathcal A_F$ is uniformly eventually sub-Markovian; furthermore, it eventually enjoys a uniform  strong Feller property. On the other hand, the semigroup on $L^2(\G)\oplus Y_d$ generated by $-\mathcal A_N$ is not even individually asymptotically positive.
 \end{proposition}

By \textit{eventual sub-Markovian} (resp., eventually irreducible) we mean that there exists some $t_0>0$ such that $0\le e^{-t\A_F}\mathfrak{f}\le {\bf 1}$ (resp., $0\ll e^{-t\A_F}\mathfrak{f}$) for all $t\ge t_0$ and all $\mathfrak{f}$ such that $0\le \mathfrak{f}\le {\bf 1}$ (resp, $0\le \mathfrak{f}$, $0\not\equiv \mathfrak{f}$), where ${\bf 1}$ is the constant 1 function. Also, a bounded semigroup is called \textit{individually asymptotically positive} if the distance between each orbit and the Hilbert lattice's positive cone tends to 0 as $t\to\infty$.

Similarly, we say that a semigroup \textit{eventually enjoys a  strong Feller property} if for all $t\ge t_0$ it is sub-Markovian and maps bounded measurable functions to bounded continuous functions.

While  $-\mathcal A_F$ generates on $L^2(\mathcal G)$ a semigroup that leaves $C(\mathcal G)$ invariant and is bounded in $\infty$-norm, it is currently unknown whether its part in $C(\mathcal G)$ is the generator of a strongly continuous semigroup.

\bibliographystyle{alpha}
\newcommand{\etalchar}[1]{$^{#1}$}

\end{document}